\documentclass[12pt,letterpaper]{article}

\usepackage{amsmath}
\usepackage{amsfonts}
\usepackage{amssymb}
\usepackage{amsthm}
\usepackage{tikz-cd}
\usepackage{graphicx}
\graphicspath{ {./images/} }
\usepackage{enumitem}
\usepackage[margin=1in]{geometry}
\usepackage{fancyhdr}
\usepackage{physics}
\newtheorem{thm}{Theorem}[section]
\newtheorem{prop}[thm]{Proposition}
\newtheorem{lem}[thm]{Lemma}

\theoremstyle{definition}
\newtheorem{defn}[thm]{Definition}

\newtheorem{remark}[thm]{Remark}

\newcommand{\R}{\mathbb{R}}
\newcommand{\N}{\mathbb{N}}
\newcommand{\D}{\mathcal{D}}
\newcommand{\E}{\mathcal{E}}
\newcommand{\F}{\mathcal{F}}
\newcommand{\M}{\mathcal{M}}
\newcommand{\X}{\mathcal{X}}
\DeclareMathOperator{\spt}{\mathrm{spt}}
\DeclareMathOperator{\Int}{\mathrm{Int}}
\DeclareMathOperator{\Haus}{\mathcal{H}}

\newcommand{\eps}{\varepsilon}
\newcommand{\Om}{\Omega}

\author{Stan Alama, Lia Bronsard, and Silas Vriend}
\title{The standard lens cluster in $\R^2$ \\ uniquely minimizes relative perimeter}
\date\today

\begin{document}

	\maketitle

    \begin{abstract}
        In this article we consider the isoperimetric problem for partitioning the plane into three disjoint domains, one having unit area and the remaining two having infinite area.  We show that the only solution, up to rigid motions of the plane, is a lens cluster consisting of circular arcs containing the finite area region, attached to a single axis, with two triple junctions where the arcs meet at 120 degree angles.  In particular, we show that such a configuration is a local minimizer of the total perimeter functional, and on the other hand any local minimizer of perimeter among clusters with the given area constraints must coincide with a lens cluster having this geometry.  Some known results and conjectures on similar problems with both finite and infinite area constraints are presented at the conclusion. 
    \end{abstract}

    \section{Introduction}

    In a previous paper \cite{alama2023free} we have conjectured that the minimal cluster for the planar partitioning problem with one unit area chamber and two improper chambers (i.e.\ of infinite area) is the standard lens cluster depicted in Figure \ref{fig:lens-cluster}. The boundary set of this cluster consists of two collinear rays emanating from a pair of standard triple junctions which are joined by a pair of symmetric circular arcs, meeting at 120 degree angles. The center corresponding to each circular arc lies on the midpoint of the opposite circular arc, and these conditions uniquely determine the geometry of the lens, also known as a  \textit{vesica piscis} (or fish's bladder.) In this paper we show that this configuration is the unique solution to the partitioning problem, in the sense that any other configuration which minimizes perimeter locally (i.e.\ among compactly supported variations preserving the area of the proper chamber) must be congruent to the unit area lens cluster.
    
    \begin{figure}[ht]
        \centering
        \includegraphics[scale=0.425]{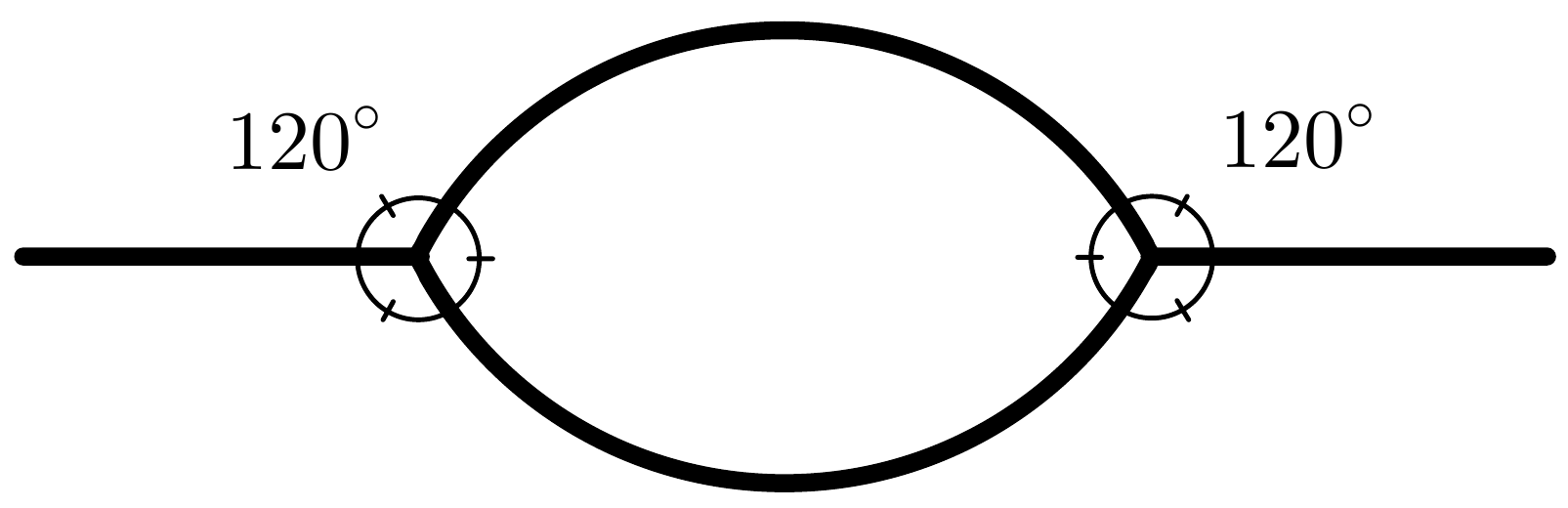}
        \caption{The standard lens cluster in $\R^2$}
        \label{fig:lens-cluster}
    \end{figure}

    All partitioning problems are extensions and generalizations of the classical isoperimetric problem: finding the shape of given area (or Lebesgue measure, in the general case) whose boundary has the smallest possible perimeter.  More recently attention has turned to partitioning problems in which the plane (or space) is subdivided into several disjoint but complementary domains, each with given measure.  This gave rise to the famous  Double-Bubble Conjecture, to determine the optimal shape of a pair of planar regions with given measure whose combined perimeter is minimized. The solution is a double-bubble composed of arcs of circles (spheres, in higher dimensions), meeting at 120 degree angles at the junction points \cite{morgan1994soap}; this was proven by a group of undergraduate students at Williams in \cite{Foisy_etal}.  An analogous question may be posed in $n$ dimensions, and other generalizations include partitions with any number of chambers of (finite) prescribed measure.  The proof that double-bubbles, domains composed of spherical caps meeting at 120 angles, minimize surface area in $\R^3$ with two given volumes, was given by Hutchings, Morgan, Ritor\'e, and Ros \cite{HutchingsMorganRitoreRos}, and generalized to any dimension by Reichardt \cite{Reichardt}.
    The case of triple-bubbles, with three prescribed areas in the plane, was handled by Wichiramala \cite{wichiramala2004proof}, but the complexity of the method has limited progress on multiple-bubble problems with more than three given areas.

    Multiple bubble configurations also arise in certain models of block copolymers: nonlocal variational problems for multiphase materials exhibiting short range attractive forces but longer range repulsion between different phases.  In an appropriate limit (described by Gamma-convergence) minimizers of a ternary model exhibit double-bubble configurations \cite{RenWei, alama2021periodic}, while a quaternary model (with four phases) produces triple-bubble minimizers \cite{ABLW2}.  

    It is in connection with the triblock copolymer models that we encountered the lens cluster problem treated in this paper.  If one considers a partitioning problem in the torus, with three phases, two of which occupy nearly all of the total area with the third accounting for only a tiny fraction, we expect minimizers of the nonlocal triblock copolymer energy to form a lamellar pattern with the two majority phases, with tiny droplets of the third phase aligned on each lamellar stripe.  By blowing up the droplets in the limit of vanishing area of the third phase we expect to recover the lens shape in the plane, and the characterization of the  \textit{vesica piscis} as the unique minimizer (up to symmetries) of this problem will be critical to the analysis of the small-area limit of the triblock problem.  (See \cite{ABLW3}.)
    
    Isoperimetric minimizers also arise as stationary or asymptotic solutions to geometrical flows.  
    In work of Bellettini-Novaga \cite{bellettini2011curvature} and Schn\"urer et al.\ \cite{schnurer2011evolution}, the lens configuration appears as the limiting configuration of a planar network with two triple junctions under curve-shortening flow.


 \subsection*{Clusters}

    In studying the  partitioning problem we adopt the notation and vocabulary of {\it clusters} as developed in Maggi \cite{maggi2012sets}.  Unlike the clusters defined in \cite{maggi2012sets}, our configurations will have two chambers with infinite measure, and so we introduce a definition which extends the framework of clusters to the more general setting in which there may be several chambers with infinite measure. 
    
    Let $\N_0$ denote the set $\N \cup \{0\}$. 
    For greater generality, we set the definitions to follow in dimension $n \geq 2$. Let $\abs{E}$ denote the Lebesgue measure of a set $E$, let $\Haus^k$ denote the $k$-dimensional Hausdorff measure, and let $P(E)$ denote the perimeter of $E$.

    \begin{defn}\label{clusterdef}
        An \textit{$(N, M)$-cluster} $\X = (\E, \F)$ in $\R^n$ is a pair of finite families of sets of {\it locally} finite perimeter 
        \begin{equation*}
            \E = \{\E(h)\}_{h = 1}^N, \quad \F = \{\F(i)\}_{i = 1}^M, \quad N, M \in \N_0,
        \end{equation*}
        which satisfies the following list of properties:
        \begin{enumerate}[label=(\roman*)]
            \item \textit{Proper chambers:} the sets $\E(h)$, $1 \leq h \leq N$, satisfy  $0 < \abs{\E(h)} < \infty$.

            \item \textit{Improper chambers:} the sets $\F(i)$, $1 \leq i \leq M$, satisfy $\abs{\F(i)} = \infty$.

            \item \textit{Null overlap:} the sets $\E(h)$ and $\F(i)$ have pairwise Lebesgue null overlap, that is,
            \begin{align*}
                \abs{\E(h) \cap \E(k)} = 0, &\quad 1 \leq h < k \leq N, \\
                \abs{\E(h) \cap \F(i)} = 0, &\quad 1 \leq h \leq N,\; 1 \leq i \leq M, \\
                \abs{\F(i) \cap \F(j)} = 0, &\quad 1 \leq i < j \leq M.
            \end{align*}
            \item \textit{Full measure:}  following  \cite{maggi2012sets} Section 30.3], we have 
            \[
                \abs{\R^n \setminus \left(\bigcup_{h = 1}^{N} \E(h)
                   \cup \bigcup_{i=1}^M \F(i)\right)} = 0.
            \]
        \end{enumerate}
    \end{defn}

    \begin{remark}\label{perimeter_remark}
        We would like to immediately highlight a key difference between this definition and those of \cite{maggi2012sets}: the proper chambers are \textit{not} assumed to have finite perimeter \textit{a priori}.  Nevertheless, we will show in Proposition~\ref{finiteper} that, for the case of $(1,2)$-clusters which minimize perimeter locally, the proper chamber must have finite total perimeter.
    \end{remark}

    We remark that as a consequence of the full measure axiom, any $(N,M)$-cluster necessarily has $M \geq 1$. Proper $N$-clusters, having $N$ disjoint chambers $\E=\{\E(h)\}_{h=1}^{N}$ of finite measure and perimeter, are defined in \cite{maggi2012sets}. Given such an $N$-cluster we may always append the exterior chamber to obtain an $(N,1)$-cluster consistent with this definition. Conversely, any $(N,1)$-cluster  whose proper chambers have finite perimeter must be an $N$-cluster as defined in \cite{maggi2012sets}, together with its exterior chamber (up to a null set). Note that {\cite{maggi2012sets} implicitly includes the complement of an $N$-cluster as an improper chamber by denoting it $\E(0)$. Note  also that Definition~\ref{clusterdef} with $N=0$ extends the notion of an improper $M$-cluster from \cite{maggi2012sets}. 
    
    It will often be convenient to denote the proper and improper chambers of a cluster with a single index, and so we also write an $(N,M)$-cluster as  $\X=\{\X(j)\}_{j=1}^{N+M}$ by setting
    \begin{equation*}
        \X(j) = \begin{cases}
            \E(j) & \text{if } 1 \leq j \leq N, \\
            \F(j - N) & \text{if } N + 1 \leq j \leq N + M.
        \end{cases}
    \end{equation*}

    \begin{remark}
        The chambers of an $(N,M)$-cluster are \textit{not} assumed to be indecomposable (i.e. measure-theoretically connected; see \cite{maggi2012sets} Exercise 16.9).  Thus, a cluster could also have several connected components in any one or more  of its chambers.
    \end{remark}



%

    \begin{defn}
        Let $\overline{\R} = \R \cup \{\pm \infty\}$ denote the extended real numbers, and let $\X = (\E, \F)$ be an $(N,M)$-cluster. The \textit{volume vector} $\mathbf{m}(\X) \in \overline{\R}^{N + M}$ is defined by
        \begin{align*}
            \mathbf{m}(\X) 
            &= (\abs{\E(1)}, \dots, \abs{\E(N)}, \abs{\F(1)}, \dots, \abs{F(M)}) \\
            &= (\abs{\E(1)}, \dots, \abs{\E(N)}, \underbrace{\infty, \dots, \infty}_{M \text{ times}}).
        \end{align*}
        In dimension $n = 2$, we call $\mathbf{m}(\X)$ the \textit{area vector}.
    \end{defn}
    
So, for example, a double-bubble cluster in $\R^2$ is described as a $(2,1)$-cluster, with two disjoint chambers $\E(1),\E(2)$ of prescribed areas  $|\E(1)|=\mathbf{m}_1=A_1$, $|\E(2)|=\mathbf{m}_2=A_2$, and the exterior chamber 
$\F=\R^2 \setminus\overline{\E(1)\cup\E(2)}$ having infinite area.  The area vector is thus $\mathbf{m}=(A_1,A_2,\infty)$.
    

To define perimeter, we set the following notation for the interfaces, in other words, the boundary arcs of the chambers of the cluster:
    \begin{defn}
        The \textit{interfaces} of an $(N,M)$-cluster $\X$ are the $\Haus^{n-1}$-rectifiable sets
        \begin{equation*}
            \X(j,k) = \partial^* \X(j) \cap \partial^* \X(k), \quad 1\leq j < k \leq N + M.
        \end{equation*}
        The \textit{boundary set} $\partial \X$ of an $(N,M)$-cluster $\X$ is defined by
        \begin{equation*}
            \partial \X = \bigcup_{1 \leq j \leq N+M} \partial\X(j).
        \end{equation*}
    \end{defn}

    We are interested in configurations which minimize perimeter, but in case the number of  improper chambers $M\ge 2$ the total perimeter will necessarily be infinite.  Hence we must express the isoperimetric problems in terms of a {\it relative perimeter}:

    \begin{defn}
        The \textit{relative perimeter} of $\X$ in $F \subset \R^n$ is defined by
        \begin{equation*}
            P(\X; F) = \sum_{1 \leq j < k \leq N + M} \Haus^{n-1}(F \cap \X(j,k)).
        \end{equation*}
        The \textit{perimeter} of $\X$ is defined by
        \begin{equation*}
            P(\X) = P(\X; \R^n) = \sum_{1 \leq j < k \leq N + 1} \Haus^{n-1}(\X(j,k)).
        \end{equation*}
    \end{defn}
    
    Similarly, while it is natural to seek a global minimizer of perimeter among all clusters with given mass vector $\mathbf{m}(\E)=\mathbf{m}$, for clusters with $M\ge 2$  improper chambers we are limited to studying {\it local minimizers,} which minimize perimeter only with respect to compactly supported perturbations of $\X$. For any set $E$ of locally finite perimeter, let $\mu_E$ denote the $\R^n$-valued Radon measure (Gauss-Green measure, see \cite[Proposition 12.1]{maggi2012sets}) for which
    $$  \int_E \text{div}\, T = \int_{\R^n} T\cdot d\mu_E, \quad \text{for all $T\in C_c^1(\R^n;\R^n).$} $$
     Then, we define:

    \begin{defn}\label{locmindef}
        Given an $(N,M)$-cluster $\X$ in $\R^n$, we say that $\X$ is a \textit{locally minimizing $(N,M)$-cluster in $\R^n$} if $\spt \mu_{\X(j)} = \partial X(j)$ for every $1 \leq j \leq N + M$, and, moreover,
        \begin{equation*}
            P(\X; B_r) \leq P(\X'; B_r)
        \end{equation*}
        for each ball $B_r$ of radius $r > 0$ centred at the origin, and every $(N,M)$-cluster $\X'$ satisfying $\mathbf{m}(\X') = \mathbf{m}(\X)$ and
        \begin{equation*}
            \X'(j) \Delta \X(j) \subset \subset B_r \quad\text{for each } 1 \leq j \leq N+M.
        \end{equation*}        
    \end{defn}



    \begin{defn}
        By an \textit{$(N,M)$-partitioning problem in $\R^n$ we mean a variational problem where given a volume vector $\mathbf{m} = (\mathbf{m}_1, \dots, \mathbf{m}_{N+M}) \in \overline{\R}^{N + M}$ with 
        \begin{align*}
            \begin{cases}
                0 < \mathbf{m}_j < \infty & \text{if } 1 \leq j \leq N, \\
                \mathbf{m}_j = \infty & \text{if } N+1 \leq j \leq N+M,
            \end{cases}
        \end{align*}
        we seek an $(N,M)$-cluster $\X$ which is a local minimizer in the sense of Definition \ref{locmindef}.}
        
    \end{defn}
    
The resolution of the double-bubble conjecture by \cite{Foisy_etal,HutchingsMorganRitoreRos}  thus proves that for any choices of $A_1,A_2>0$, the $(2,1)$-partitioning problem in $\R^n$ with volume vector $\mathbf{m}=(A_1,A_2,\infty)$ is a double-bubble, consisting of spherical interfaces meeting at 120 degree angles at the junctions.  Such a configuration is both a global and local minimizer of the perimeter.

%

\subsection*{The Main Result}

    \begin{defn}\label{lensdef}
        The \textit{standard lens cluster} in $\R^2$ is the 
        $(1,2)$-cluster $\X_\infty = (\E,\F)$ 
        whose chambers are given by
        \begin{gather*}
            \E(1) = \{(x,y) : \abs{x} \leq \tfrac{\sqrt{3}}{2} R,\; \abs{y} \leq u(x)\}, \\
            \F(1) = \{y \geq 0\} \setminus \Int(\E(1)), \quad
              \F(2) = \{y \leq 0\} \setminus \Int(\E(1)),
        \end{gather*}
        where $u : [-\frac{\sqrt{3}}{2} R,\frac{\sqrt{3}}{2} R] \to \R$ is determined by 
        \begin{align*}
            u(x) &= \sqrt{R^2 - x^2} - \frac{1}{2} R,
        \end{align*}
        and the radius of curvature $R$ is determined by 
        \begin{equation}\label{Rdef}
            R^{-1} = \sqrt{\frac{2\pi}{3} - \frac{\sqrt{3}}{2}}.
        \end{equation}
        An elementary computation shows that $\abs{\E(1)} = 1$.
    \end{defn}   
    Note that the two interfaces of $\E(1)$ are circular arcs of the same radius $R$, and the center of each arc is located on the midpoint $(0,\pm{\tfrac{R}{2}})$ of the opposite arc.

        In this article, we prove the 
    \begin{thm}[Main Result]  \label{mainresult}
        The standard lens cluster is (up to translation and rotational symmetry) the unique locally minimizing cluster for the   $(1,2)$-partitioning problem in $\R^2$, with area vector $\mathbf{m}=(1, \infty, \infty)$.
    \end{thm}

The proof of Theorem~\ref{mainresult} will be in two parts.  First, we show that the standard lens cluster is a local minimizer, in the sense of Definition~\ref{locmindef}.  For this, we use a strategy suggested to us by Frank Morgan, thinking of the lens cluster as a limit of double-bubbles as the area of one chamber diverges to infinity.  This limit is described in  Section 2.  Next, we show that any local minimizer must reduce to the standard lens cluster, by a sequence of constructions which would reduce the local perimeter were the configuration to have a different geometry.  The proof of Theorem~\ref{mainresult} is contained in Section 3.   Some conjectures on the geometry of minimizing $(N,M)$-clusters for larger values of $M,N$ are included in  Section 4.
    
    \section{A Convergence Result}

In this section we show that the standard lens cluster may be obtained as a local limit of double bubbles with a very large chamber.  We first define a weak notion of convergence for clusters.

    \begin{defn}\label{distancedef}
        Let $N,M,N',M'$ be such that $N + M = N' + M'$. The \textit{cluster distance} in $F \subset \R^n$ between an $(N,M)$-cluster $\X$ and an $(N', M')$-cluster $\X'$ in $\R^n$ is defined as
        \begin{equation*}
            d_F(\X, \X') = \sum_{j = 1}^{N + M} \abs{F \cap (\X(j) \Delta \X'(j))}.
        \end{equation*}
        We set $d(\X, \X') = d_{\R^n}(\X, \X')$.  Notice that in general we could have $d(\X, \X') = \infty$.
    \end{defn}

    \begin{defn}
        \, Fix $N,M$ and let $\X$ be an $(N', M')$-cluster such that $N' + M' = N + M$.
        \begin{itemize}
            \item We say that a sequence of $(N, M)$-clusters $\{\X_{k}\}_{k \in \N}$ in $\R^n$ \textit{locally converges to $\X$}, and write $\X_k \xrightarrow{\mathrm{loc}} \X$, if for every compact set $K \subset \R^n$ we have $d_K(\X_k, \X) \to 0$ as $k \to \infty$.

            \item We say that $\{\X_k\}_{k \in \N}$ \textit{converges to $\X$}, and write $\X_k \to \X$, if $d(\X_k, \X) \to 0$ as $k \to \infty$.
        \end{itemize}
    \end{defn}

    \begin{remark}
        For clusters with $M \geq 2$, we can in general only talk about local convergence. Note that cluster type need not be preserved under cluster convergence, and cluster distance is indexing-dependent.
    \end{remark}

%
%

    The key idea in the existence part of the proof of the main result is that the standard lens cluster is the limit in local cluster convergence of a suitably chosen family of standard double bubbles (with their exteriors). We make this precise below.

    \begin{lem}
        If $\X$ is an $(N,M)$-cluster in $\R^n$, then for every $F \subset \R^n$ we have
        \begin{equation*}
            P(\X; F) = \frac{1}{2} \sum_{1 \leq j \leq N+M} P(\X(j); F).
        \end{equation*}
        In particular, if $U$ is open in $\R^n$ and $\X_k \xrightarrow{\mathrm{loc}} \X$, then
        \begin{equation*}
            P(\X; U) \leq \liminf_{k \to \infty} P(\X_k; U).
        \end{equation*}
    \end{lem}

    \begin{proof}
        This follows from a straightforward adaptation of \cite{maggi2012sets} Proposition 29.4. 
    \end{proof}

    \begin{lem}\label{lem:approximating-double-bubble}
        For $A > 0$, consider the minimizing double bubble $\D_A = \{\D_A(1), \D_A(2), \D_A(3)\}$, $\D_A(3)=\R^2\setminus\overline{\D_A(1)\cup\D_A(2)}$, with area vector $(1,A, \infty)$, oriented so that the circle supporting the interface  $\D_A(2,3)$ is tangent to the $x$-axis at the origin  and centered on a point of the positive $y$-axis; see Figure \ref{fig:approximating-double-bubble}. Then as $A \to \infty$,
        \begin{equation*}
            \D_A \xrightarrow{\mathrm{loc}} \X_\infty.
        \end{equation*}
        Furthermore, for any bounded open set $U$, 
        \[
            P(\X_\infty; U) \le \liminf_{A \to \infty} P(\D_A; U).
        \]
    \end{lem}

    \begin{proof}
            Fix $r > \tfrac{\sqrt{3}}{2}R$. By monotonicity of $d_F$ with respect to $F$ and boundedness of each compact set, it suffices to check that for each open ball $B_{r'}$ centered at $0$ with $r' \geq r$, we have 
        \begin{equation*}
            d_{B_{r'}}(\D_A, \X_\infty) \xrightarrow[A \to \infty]{} 0.
        \end{equation*}

        Let $r' \geq r$ be arbitrary. We aim to derive estimates on $d_{B_{r'}}(\D_A, \X_\infty)$ via geometrical information about $\D_A$ (see \cite{isenberg}).
        A double bubble is bounded by three circular arcs $C_0,C_1,C_2$ of radii $r_1, r_2$ and $r_0$, where $r_0$ is the radius of the common boundary of the two lobes of a double bubble with area vector $(m_1,m_2, \infty)$. Denote by $\theta_1, \theta_2$, and $\theta_0$ the angles associated with the three arcs, as in Figure \ref{fig:double-bubble}. The double bubble $\D_A$ corresponds to the case $m_1=1, m_2=A$.  We have fixed translation invariance by selecting $\D_A$ so that the circle $C_2$ supporting the interface  $\D_A(2,3)$ is tangent to the $x$-axis at the origin.
        Note this implies that in the limit $A\to\infty$ the large circle $C_2$ will approach the $x$-axis.
        By axial symmetry we conclude that the centers $p_0, p_1, p_2$ of the three circles each lies on the $y$-axis. We denote by $q$ either one of the triple junction points (which are placed symmetrically about the $y$-axis) where the three circular arcs meet.

        \begin{figure}[ht]
            \centering
            \includegraphics[scale=0.47]{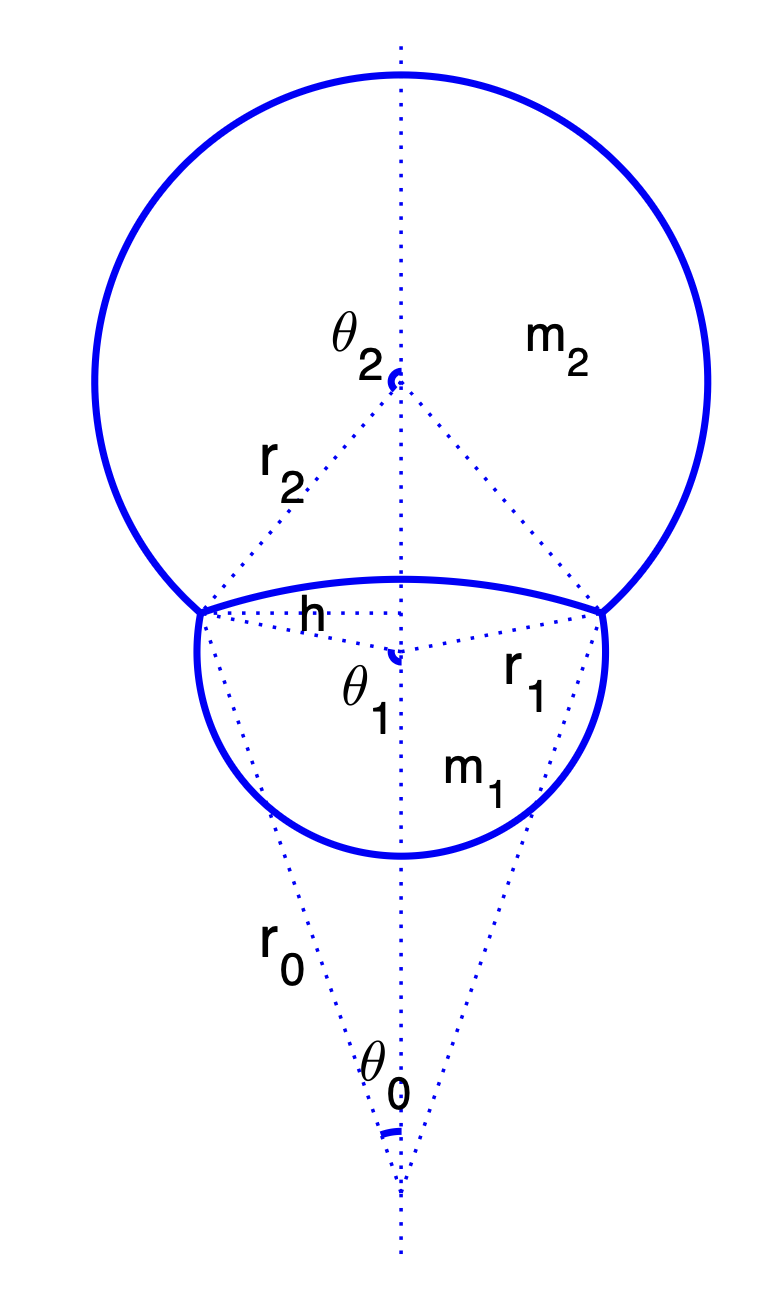}
            \caption{An asymmetric double bubble in $\R^2$ \cite{alama2021periodic}}
            \label{fig:double-bubble}
        \end{figure}
        
        The equations relating areas, angles, and radii are 
        \begin{eqnarray}
            \label{e1A}
            r_1^2 (\theta_1 - \cos \theta_1 \sin \theta_1) + r_0^2 (\theta_0 - \cos \theta_0 \sin \theta_0)  &=& 1, \\
            \label{e2A}
            r_2^2 (\theta_2 - \cos \theta_2 \sin \theta_2) - r_0^2 (\theta_0 - \cos \theta_0 \sin \theta_0)  &=& A, \\
            \label{e3A}
            r_1 \sin \theta_1 &=& r_0 \sin \theta_0, \\
            \label{e4A}
            r_2 \sin \theta_2 &=& r_0 \sin \theta_0, \\
            \label{e5A}
            (r_1)^{-1} - (r_2)^{-1} &=& (r_0)^{-1},\\
            \label{e6A}
            \cos \theta_1 + \cos \theta_2 + \cos \theta_0 &=& 0.
        \end{eqnarray}
        As $A\to \infty$, $r_2\to \infty$ and  $\theta_2\to \pi$. Hence from equations (\ref{e5A}) and (\ref{e3A}), we see that
        \begin{equation}\label{rdiff}
        |r_1-r_0|\to 0
        \end{equation}
        and 
        \begin{equation}\label{thetadiff}
        |\theta_1-\theta_0|\to 0.
        \end{equation}
        
        Since $\D_A$ is a minimizing double bubble, the tangents to the arcs meet at angles of $2\pi/3$.  Now form the triangle $\triangle p_0 p_2 q$.  The edges $\overline{p_0 q}$ and $\overline{p_2 q}$ are radii of the circular arcs, and thus are normal lines to the transition curves at $q$.  They must therefore meet at angle $2\pi/3$, and so adding the interior angles of the triangle $\triangle p_0 p_2 q$ we obtain $\theta_0=\theta_2- {\frac{2\pi}{3}}$.  As a consequence, we conclude that $\theta_0$ is convergent, $\theta_0\to{\frac{\pi}{3}}$ as $A\to\infty$. From \eqref{thetadiff} we also have  $\theta_1\to{\frac{\pi}{3}}$.  

        To pass to the limit in the radii, we note that $|\D_A(1)|=1$, and so $0<r_1<1/\sqrt{\pi}$ holds for all $A$.  By \eqref{rdiff} both $r_0, r_1$ are bounded, so there exists a subsequence along which both converge, to the same limiting value.  Using equation (\ref{e1A}) we conclude that $r_1^{-1}\to
        \sqrt{2\pi/3-\sqrt{3}/2}=:R^{-1}$ which yields the radius of the circles in the standard lens cluster $\X$. From \eqref{rdiff} we also have $r_0^{-1} \to R^{-1}$. 

        To show convergence of the centers of the circles $C_0, C_1$ we first note that the large circle $C_2$ converges locally uniformly to the $x$-axis
        in the sense of parametrized curves. As the centers of $C_0,C_1$ must lie on the $y$-axis, and the radii $r_0,r_1$ converge, the triple point $q$ remains bounded, and since it also lies on $C_2$, passing to a further subsequence if necessary, it converges to a point on the $x$-axis (see Figure \ref{fig:endpf}). The center $p_0$ (resp.\ $p_1$) must lie on the intersection of the $y$-axis with the circle of radius $r_0$ (resp.\ $r_1$) centered at $q$, which consists of two points. In the limit, $r_0,r_1\to R$ and the centers converge to the two points lying on the intersection of the $y$-axis and the circle of radius $R$ centered at the limiting triple point $q$. The same is true of the circle of radius $R$ centered at the opposing triple point. The centers $p_0,p_1$ lie at the intersection of these two circles, and are symmetrically placed with respect to the line connecting the two triple points.  Now consider the isosceles triangle $\triangle p_0 p_1 q$.  Again, the radius vectors of the circles are normal vectors to the curves.  As the tangents meet at angle $2\pi/3$, the angle $\angle p_0 q p_1$ at the triple point is $\pi/3$, and therefore  $\triangle p_0 p_1 q$ is equilateral, of length $R$.  In particular, the segment joining $p_0$ and $p_1$ has length $R$ and we conclude that $p_0=(0,-R/2)$, $p_1=(0,R/2)$ and $q=\frac{\sqrt{3}}{2} R$.

        \begin{figure}
            \centering
            \includegraphics[scale=1.2]{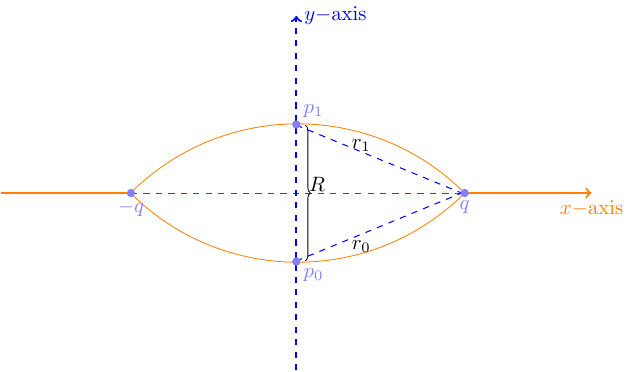}
            \caption{Convergence to the standard double-bubble}
            \label{fig:endpf}
        \end{figure}

        Finally, the boundary arcs being circular, they represent parametrized piecewise $C^\infty$ arcs.  We have already noted that the large circle $C_2$ converges locally uniformly to the $x$-axis.  Since the radii and centers of $C_0,C_1$ converge, each circle converges in the sense of parametrized curves to a limiting circle, and so each of the arcs in the double bubble $\D_A$ converges, locally uniformly in the sense of parametrized curves, to a circular arc. The local uniform convergence then implies the convergence of the associated clusters in the distance defined by Definition~\ref{distancedef}. 
        
        Since  $r' \geq r > \tfrac{\sqrt{3}}{2}R$, the lens is inside $B_{r'}$ and the conclusion follows.
    \end{proof}

%
%

    \section{Proof of Theorem \ref{mainresult}}
    
We now prove that the standard lens cluster is the unique locally minimizing $(1,2)$-cluster, up to rigid motions of the plane, with area vector $\mathbf{m}=(1,\infty,\infty)$.

\subsection*{Minimality of  $\X_\infty$}

We first show existence of a local minimizer:  in particular, that  $\X_\infty$ as defined in Definition~\ref{lensdef} is a local minimizer of the perimeter functional.
The strategy is to argue by contradiction and assume that there is a compactly supported perturbation of the lens cluster which decreases perimeter.  If so, we construct a competitor for the double-bubble problem with a large chamber, which would have less perimeter than a minimizing double bubble.  The argument relies on the following corollary of the double bubble theorem:

    \begin{thm}[\cite{Foisy_etal}]\label{lem:given-or-greater-area}
    Fix $A>0$, and consider the family of planar partitioning problems,
        \begin{equation*}
            \inf \{ P(\E) : \mathbf{m}(\E) = (1, A',\infty), \ A'\geq A\},
        \end{equation*}
        where $\E$ ranges over all planar $(2,1)$-clusters with area vector $(1, A',\infty)$. Then, up to translation and rotation, the unique minimizer of this problem is the standard double bubble $\D_A$ with area vector $(1,A,\infty)$. 
   \end{thm}

        From now on,  we will refer to the standard double bubble  $\X_\infty$ as $\X$. Assume toward a contradiction that $\X$ is \textit{not} a locally minimizing $(1,2)$-cluster for the area vector $(1, \infty, \infty)$.  Negating Definition \ref{locmindef}, it follows that there exists an open ball $B_r$ of radius $r > 0$ centered at $0$ and a  $(1,2)$-cluster $\X'$ such that the area vectors of $\X=\X_\infty$ and $\X'$ agree, 
        \[
            \mathbf{m}(\X') = \mathbf{m}(\X),
        \]
        the modifications are compactly contained in $B_r$, 
        \begin{equation}\label{eq:symmetric-differences}
            \X(j) \Delta \X'(j) \subset \subset B_r \text{ for each } 1 \leq j \leq 3,
        \end{equation}
        and $\X'$ improves on $\X$ within $B_r$,
        \begin{equation*}
            P(\X'; B_r) < P(\X; B_r).
        \end{equation*}
        
        
         By \eqref{eq:symmetric-differences}, it follows that $\X(1), \X'(1) \subset \subset B_r$. Note that we make no \textit{a priori} assumptions about the regularity of the interfaces of $\X'$ within $B_r$. Denote by $K$ a compact set such that
        \begin{equation*}
            \X(j) \Delta \X'(j) \subset K \subset B_r \quad \text{for each } 1 \leq j \leq 3.
        \end{equation*}
        By enlarging $K$ if necessary, we can assume that $K$ is a closed ball centered at $0$ containing $\X(1)$ and $\X'(1)$ in its interior.


        Since $\mathbf{m}(\X') = \mathbf{m}(\X)$, we have that $\abs{\X'(1)} = 1$. One of $\X'(2) \cap B_r$ or $\X'(3) \cap B_r$ must have the same or strictly less relative area within $B_r$ compared to $\X(2) \cap B_r$ or $\X(3) \cap B_r$, respectively. As $\X$ is symmetric about the $x$-axis, without loss of generality we can assume that the upper region $\X'(2) \cap B_r$ has area at least that of $\X(2) \cap B_r$. 
        Since $\X$ and $\X'$ only differ in $K$, we can assume,
        \begin{equation}\label{*}
            \abs{\X'(2) \cap K} \geq \abs{\X(2) \cap K}.
        \end{equation}
        Set 
        \begin{equation}\label{eps}\eps = P(\X; B_r) - P(\X'; B_r) > 0.
        \end{equation}

        Consider the one-parameter family of $(2,1)$-clusters $\D_A$ associated with the minimizing double-bubbles defined in Lemma \ref{lem:approximating-double-bubble}, with area vectors $(1, A, \infty)$; see Figure \ref{fig:approximating-double-bubble} for a diagram.  By Lemma~\ref{lem:approximating-double-bubble}, we know that $\D_A$ locally converges to $\X$ in the sense of cluster distance as $A \to \infty$, i.e.,
        \begin{equation*}
            \D_A \xrightarrow{\mathrm{loc}} \X,
        \end{equation*}
        and furthermore that
        \begin{equation*}
            P(\X; U) \le \liminf_{A \to \infty} P(\D_A; U)
        \end{equation*}
        for any bounded open set $U$. Note that for all $A$ sufficiently large, $\partial \D_A(2) \cap B_r$ consists of piecewise smooth circular arcs which are graphs over the $x$-axis. By construction of $\D_A$, we fix translation invariance so that the large circle $C_2$  is tangent to the axis at the origin. Note in fact that the graphs forming $\partial \D_A(2) \cap B_r$ are positive functions.

        \begin{figure}
            \centering
            \includegraphics[scale=0.425]{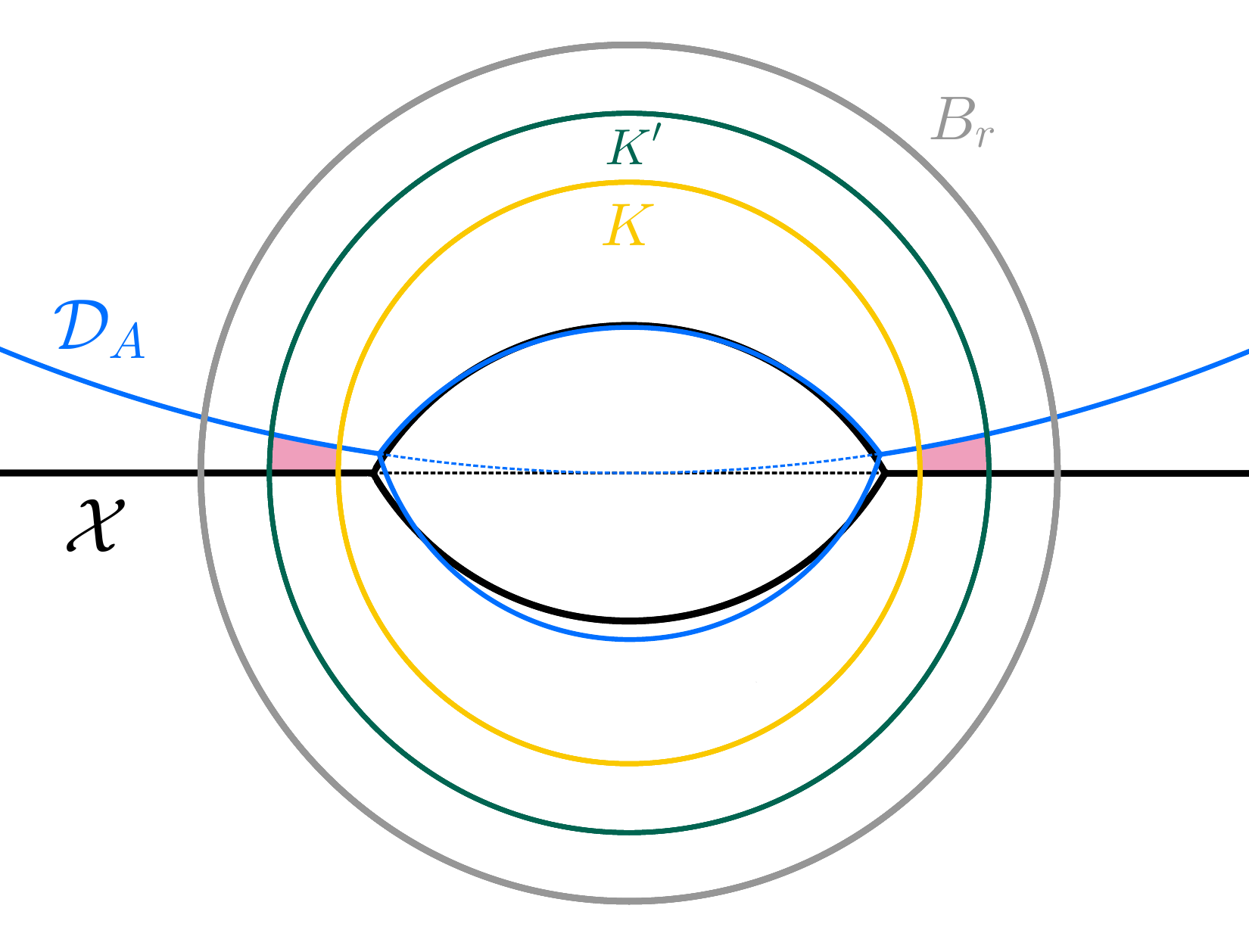}
            \caption{The $(2,1)$-cluster $\D_A$ (in blue) locally approximates the standard lens cluster $\X$ (in black) for large $A$. The lens $\X(1)$ is contained in nested sets $K$ (in yellow), $K'$ (in green), and $B_r$ (in grey). The region $\mathcal{R}_A$ to be added to \ $\D_A(2)$ is shaded (in pink)} 
            \label{fig:approximating-double-bubble}
        \end{figure}

        Let $K'$ be a closed ball concentric with $K$ and $B_r$ such that $K \subset K' \subset B_r$. 
        Define the open annulus $\mathcal{A} = \Int(K') \setminus K$ and the set $\mathcal{R}_A = \mathcal{A} \cap (\X(2) \setminus \D_A(2))$.
        We consider two cases. First suppose that for $ A$ sufficiently large, the area of $\D_A(2)\cap K$ is smaller than the area of $\X(2)\cap K$, that is, for some $A_0>0$ we have 
        \begin{equation}\label{**}
            \abs{\D_A(2) \cap K} \leq \abs{\X(2) \cap K}, \ \text{for all } A\geq A_0.
        \end{equation}

        In this case, we use the variation $\X'$ to define a competitor $\D'_A$ to the double bubble $\D_A$ in the following way:  $\D'_A$ coincides with $\X'$ in $K$, and agrees with the double-bubble $\D_A$ outside of $K'$.  In the annular region $\mathcal{A}=K'\setminus K$ we connect the two by taking $\D'_A(2)=\D_A(2)\cup \mathcal{R}_A$ while $\D'_A(3)=\D_A(3)\setminus \mathcal{R}_A$ (see Figure \ref{fig:approximating-double-bubble}). 
         (Note that the area of the chamber $\D_A'(2)$ may not be exactly equal to $A$, but we will see below that it is at least as large.) In constructing this competitor, we have modified the perimeter in a predictable way, namely we have added the small arc $\partial \mathcal{R}_A\cap \partial K'$ and replaced the arc $\D_A(2)\cap \mathcal{A}$ with  the segments of the $x$-axis contained in the annular region $\mathcal{A}$. By choosing $A$ sufficiently large, we can make this smaller than $\frac{\eps}2$ with $\eps$  as defined by (\ref{eps}) and obtain:  
        \begin{align*}
            P(\D_A'; B_{r})&\leq P(\X';B_{r})+ \frac{\eps}2\\
            &\leq P(\X; B_r)-\eps + \frac{\eps}2\\
            &< \liminf_{A\to \infty} P(\D_A; B_r).
        \end{align*}

        This means that we can find some area $A\geq A_0$ for which $P(\D_A'; B_r)<P(\D_A; B_r)$. On the other hand, we have also modified the area of this competitor in a predictable way, and using (\ref{*}) and (\ref{**}) we find: 
        \begin{align*}
            \abs{\D_A'(2)}&=\abs{\D_A'(2)\cap K}+\abs{\D_A'(2)\cap \mathcal{A}}+\abs{\D_A'(2)\setminus {K'}}\\
            &\geq \abs{\X'(2) \cap K}+ \abs{\D_A(2)\cap \mathcal{A}} +\abs{\D_A(2)\setminus{K'}}\\
            &\geq \abs{\X(2) \cap K} +\abs{\D_A(2)\setminus K} \\
            &\geq \abs{\D_A(2)\cap K} +\abs{\D_A(2)\setminus K}\\
            &=\abs{\D_A(2)} = A.
        \end{align*}

        This contradicts the minimality of the double-bubble $\D_A$ as per Theorem \ref{lem:given-or-greater-area} and thus (\ref{**}) does not hold.

        In case (\ref{**}) fails, we can find a sequence $A_n\to \infty$ for which:
        \begin{equation}\label{***}        \abs{\D_{A_n}(2)\cap K}> \abs{\X(2)\cap K}.
        \end{equation}

        Consider the vertical translates $\D_{A_n}+(0,t)$ of $\D_{A_n}$. As $\partial \D_{A_n}(2)\cap B_r$ consists of graphs of positive functions, the area $\abs{(\D_{A_n}(2)+(0,t))\cap B_r}$ is continuous and strictly monotone decreasing to zero as $t\to \infty$. Thus we can define translates $t_n>0$ for which: 
        \begin{equation} \label{****}
        \abs{(\D_{A_n}(2)+(0,t_n))\cap K} = \abs{\X(2)\cap K}, \ \text{for all } n.
        \end{equation}

        Next we claim that $t_n\to 0$. Indeed, suppose not, so that extracting a subsequence if necessary, we may assume that $t_n\to t>0$.
        Observe that 
        \begin{equation*}
            \abs{(\D_{A_n}(2)+(0, t_n/2)\cap K}>\abs{(\D_{A_n}(2)+(0, t_n)\cap K}=\abs{\X(2)\cap K}, \ \text{for all } n.
        \end{equation*}

        Therefore by passing to the limit we have $\D_{A_n}(2)\cap K\to \X(2)\cap K$ locally and hence we conclude that $\abs{(\X(2)+ (0, t/2)\cap K}\geq \abs{\X(2)\cap K}$ which contradicts the strict monotonicity of the area under vertical translation. Hence $t_n\to 0$ and therefore we may also conclude from Lemma \ref{lem:approximating-double-bubble} that $\D_{A_n}+(0, t_n)\to \X$ locally.
        
        Now construct the competitor as before but replacing $\D_A$ with $\D_{A_n}+(0, t_n)$, namely $\D_{A_n}''$ agrees with $\X'$ in $K$, coincides with $\D_{A_n}+(0, t_n)$ outside of $K'$, and with $\D''_{A_n}(2)={\mathcal{{R}}}_{A_n}\cup (\D_{A_n}(2)+(0, t_n))$ and $\D''_{A_n}(3)=(\D_{A_n}(3)+(0,t_n))\setminus \mathcal{R}_{A_n}$ in $\mathcal{A}$. 
        As before the modification of the perimeter only occurs in the annular region and by taking the area $A_n$ sufficiently large, the modification can be made smaller than $\frac{\eps}2$ where $\eps$ is defined  by (\ref{eps}). 
        So the perimeter of the competitor now satisfies:
        \begin{align*}
            P(\D_{A_n}''; B_r)
            &\leq P(\X;B_r)-\frac{\eps}2\\
            &< \liminf P(\D_{A_n}+(0, t_n);B_r).
        \end{align*}
        Further using (\ref{*}) and (\ref{****}) the area of the competitor satisfies: 
        \begin{align*}
            \abs{\D_{A_n}''(2)}&= \abs{(\D_{A_n}(2)+(0,t_n))\setminus K'} + \abs{{\mathcal{{R}}}_{A_n}\cup (\D_{A_n}(2)+(0,t_n)) \cap \mathcal{A}} + \abs{\X'(2)\cap K} \\
            &>\abs{(\D_{A_n}(2)+(0,t_n))\setminus K' }+\abs{(\D_{A_n}(2)+(0,t_n)) \cap \mathcal{A}}+\abs{\X(2)\cap K}\\
            &=\abs{(\D_{A_n}(2)+(0,t_n))\setminus K' }+\abs{(\D_{A_n}(2)+(0,t_n)) \cap \mathcal{A}}+\abs{(\D_{A_n}(2)+(0, t_n))\cap K}\\
            &\geq \abs{\D_{A_n}(2)+ (0, t_n)}.
        \end{align*}

        This once again contradicts the minimality of the double bubble with area $(1,A_n)$ for $A_n$ sufficiently large.  Therefore the lens cluster $\X$ is a minimizer in the mixed planar partitioning problem with area vector $(1, \infty, \infty)$.

        \subsection*{Uniqueness of the minimizer}

        Let $\M = (\M(1), \M(2), \M(3))$ be a locally minimizing $(1,2)$-cluster for the mixed planar partitioning problem with area vector $(1, \infty, \infty)$. We will show that (up to symmetries of the problem) necessarily  $\M = \X$, the standard lens cluster. 
        We recall the notation $\M(j,k)=\partial^* \M(j) \cap \partial^* \M(k)$ for the interfaces between the chambers.

        First, we recall Remark~\ref{perimeter_remark}, and prove that a locally minimizing $(N,M)=(1,2)$ cluster $\M$ (which {\it a priori} is defined as having locally finite perimeter in each chamber) must in fact have a proper chamber $\M(1)$ with finite total perimeter in $\R^2$.

        \begin{prop}\label{finiteper}  Suppose $\M=(\M(1),\M(2),\M(3))$ is a cluster in $\R^2$ satisfying:
            \begin{enumerate}[label=(\roman*)] 
                \item $|\M(1)|=1$;
                \item $|\M(2)|=\infty=|\M(3)|$:
                \item  Each $\M(j)$ has locally finite perimeter;
                \item  $\M$ is locally minimizing.
            \end{enumerate}
        Then $\mathrm{Per}_{\R^2}(\M(1))<\infty$.
        \end{prop}
        \begin{proof}  Assume for a contradiction that $\mathrm{Per}_{\R^2}(\M(1))=\infty$.  Then we may choose $R>0$ for which 
        \begin{equation}\label{FP1}   \mathrm{Per}_{B_R}(\M(1))\ge 3\sqrt{\pi}.
        \end{equation}
        As
        $$  1=|\M(1)| = \int_0^\infty \int_{\partial B_r} \mathbf{1}_{\M(1)}\, ds_r\, dr,  
        $$
        by Fubini's Theorem there exists $R_0\in [R,R+1]$ so that 
        \begin{equation}\label{FP2}  \int_{\partial B_{R_0}} \mathbf{1}_{\M(1)}(x)\, ds_{R_0} \le 1
             \quad \text{ie.,} \quad \Haus^1\left(\M(1)\cap \partial B_{R_0}\right) \le 1,    
        \end{equation}
        and thus $\M(1)$ intersects $\partial B_{R_0}$ on arcs with total arclength at most one.  Define
        $$ \eps:= \int_{B_{R_0}} \mathbf{1}_{\M(1)} (x)\, dx \le 1.  $$
        Inside $B_{R_0}$, $\partial \M(1)$ must intersect either $\partial\M(2)$ or $\partial\M(3)$ (or both).  Suppose it's true for $\partial\M(2)$, so $\Haus^1(\partial\M(1)\cap\partial\M(2)\cap B_{R_0})>0$. 

        Now, construct a competitor $\tilde\M$ as follows, by modification of $\M$ inside $\overline{B_{R_0}}$. Choose $r=\sqrt{\eps/\pi}$ so that the ball $|B_r(0)|=\eps$.  Then, we define
        \begin{gather*}
            \tilde \M(1):= \left[\M(1)\setminus B_{R_0}\right] \cup B_r(0), \\
            \tilde \M(2):= \left( \M(2)\cup [\M(1)\cap B_{R_0}] \right)\setminus B_r(0), \\
            \tilde \M(3):= \M(3)\setminus B_r(0).
        \end{gather*}
        Outside of $B_{R_0}$, nothing changes.  The value of $r$ is chosen so that the area constraint $|\tilde\M(1)|=1$ is satisfied.  Inside $B_{R_0}$, perimeter is decreased by gluing $\M(1)\cap B_{R_0}$ onto $\M(2)$, as some interfaces will disappear.  On the other hand, new interfaces may be created along $\partial B_{R_0}$, but by \eqref{FP2} the added perimeter is bounded by one.  We may thus calculate the change in perimeter of the proper chamber:  by \eqref{FP1}, \eqref{FP2}, and the choice of $r$,
        \begin{align*}
            \mathrm{Per}_{B_{R_0}}\left(\M(1)\right) - \mathrm{Per}_{B_{R_0}}\left(\tilde\M(1)\right)
              &\ge 3\sqrt{\pi} - \mathrm{Per}_{B_{R_0}}(B_r) - \Haus^1\left(\M(1)\cap\partial B_{R_0}\right) \\
              &\ge 3\sqrt{\pi} - 2\sqrt{\pi} - 1 = \sqrt{\pi}-1>0.
        \end{align*}
        We conclude that the perimeter $P(\tilde\M; B_{R_0})<P(\M;B_{R_0})$, which contradicts the local minimality of $\M$.  Hence $\mathrm{Per}_{\R^2}(\M(1))<\infty$.
        \end{proof}
        
        We now prove the uniqueness statement via a series of claims.

        \begin{enumerate}

            \item  {\it Each interface $\M(j,k)$ is an analytic arc of constant curvature, and so either a straight line or an arc of a circle.}  This follows from Almgren \cite{almgren} (see also p. 391 or Remark~30.4 in \cite{maggi2012sets}, and Taylor \cite{taylor} for the three-dimensional case.)  Note that Almgren and Maggi only consider the case of a single  improper chamber, but the regularity theory holds locally along the interfaces and so the proofs are identical.  

            Since the chambers of $\M$ have piecewise smooth boundary, we may assume that each is an open set, and speak of its connected components (as opposed to indecomposable measurable sets, the proper notion of connectivity for sets of locally finite perimeter.)
            
            \item {\it The transitions $\M(2,3)$, between phase domains $\M(2)$ and $\M(3)$ are straight line segments,  rays, or full lines.}  As there is no mass constraint for $\M(2),\M(3)$, there is no Lagrange multiplier in the Euler-Lagrange equation, and so minimizers must have zero curvature at points of $\M(2,3)$.
            
            \item {\it The transitions $\M(1,2),\M(1,3)$ between the finite area component $\M(1)$ and the infinite area components $\M(2),\M(3)$ are circular arcs, each of the same constant curvature independent of the connected component of $\M(1)$ or the neighboring phase $\M(2), \M(3)$.}  This follows from the regularity theory, and the calculation of the Euler-Lagrange equation in the area constrained case, as in \cite[Theorem 17.20]{maggi2012sets}. The choice of Lagrange multiplier should be the same at any regular point in $\partial \M(1)$, and so is independent of the connected component or of the choice of interface $\M(1,3),\M(1,2)$.  
            
            \item  {\it Singular points of $\M(j,k)$ occur when analytic arcs meet; these can only occur at triple junction points, that is, points at which all three phases meet, and the tangents to the arcs form 120 degree angle junctions.}  Again, this follows from the regularity theory, and the classification of planar cone clusters \cite[Theorem 30.7]{maggi2012sets}.  As an immediate consequence, we observe that the lines forming the transitions $\M(2,3)$ cannot cross, and can only terminate on components of the finite area domain $\M(1)$.  
            \item \label{unbounded} {\it  The infinite area components $\M(2),\M(3)$ need not be connected, but they cannot have bounded connected components.}  Assume that (say) $\M(2)$ has a connected component $\Omega$ which is bounded.  Suppose $\partial\Om$ has nontrivial intersection with $\partial\M(3)$.  In that case, we create a new cluster, $\tilde\M=(\tilde \M(1),\tilde \M(2),\tilde \M(3))$ which is identical to $\M$ except that $\Omega$ is assigned to $\tilde \M(3)$, that is, $\tilde \M(3)=\M(3)\cup \Omega$, $\tilde \M(2)=\M(2)\setminus\overline{\Omega}$, and $\tilde \M(1)=\M(1)$. Then, $\tilde\M$ is an admissible cluster which agrees with $\M$ outside of a compact set (containing $\Omega$), but has smaller total perimeter, as a component of $\M(2,3)$ is erased by the modification.  This contradicts local minimality of $\M$, and so we obtain the desired conclusion.  If instead the boundary $\partial\Om$ is disjoint from $\partial\M(3)$, then the component $\Om$ must be contained inside a connected component $\tilde\Om$ of $\M(1)$. If there were such an island, by translating it inside $\tilde\Om$ until it contacts the outer boundary of $\tilde\Omega$ we create a competing configuration with the same area of $\M(1)$, which agrees with $\M$ outside of a compact set, and with the same perimeter.  However, the new configuration has a singular point where the phases are tangent, violating the 120 degree condition at junctions.  
            
            \item \label{nequalsone}  {\it Each connected component of $\M(1)$ has exactly two triple junction points on its boundary. }  That is, for any connected component $\Omega$ of $\M(1)$, $\partial \Omega$ consists of two disjoint arcs of circles (of the same curvature, for each arc and each connected component), one from $\M(1,2)$ and one from $\M(1,3)$.  
            
            The proof of this claim is subtle, and will be done in several steps.  First, any connected component $\Omega$ of $\M(1)$ must be bounded.  This follows from our definition of improper cluster, in particular that $\M(1)$ is a set of (globally) finite perimeter.  In the plane, the perimeter controls the diameter of a connected component, which must therefore by finite.
            
            Next, any connected component $\Omega$ of $\M(1)$ must contact both of $\M(2),\M(3)$ on its boundary.  First, note that by the previous paragraph $\Omega$ cannot be multiply connected, with an island of either $\M(2)$ or $\M(3)$ inside. Next suppose that $\M$ is a local minimizer with a component $\Omega$ whose boundary is disjoint from (say) $\M(3)$. By translating $\Omega$ inside $\M(2)$ until it is tangent to a regular point of $\E(2,3)$ we create an admissible variation of $\M$ with the same area constraint and the same perimeter, but with a singular point which violates the 120 degree angle condition at junctions between the three phases. Hence it is impossible to have a connected component of $\M(1)$ without triple junctions.  As the outer boundary of $\Omega$ is a closed curve, for each transition of $\M(2)$ to $\M(3)$ on the curve there must be a second transition with the opposite orientation, and so at least a pair of junctions on $\partial\Omega$.

            Now suppose there are $2n$ transitions on $\partial\Omega$; we will show that necessarily $n=1$.   Consider a large circle $C_R$ of radius $R$ which contains the component $\Omega$ inside.  Since each connected component of $\M(2)$, $\M(3)$ is unbounded, by choosing $R$ sufficiently large $C_R$ must intersect both unbounded sets, and so there must be at least two transitions between the unbounded phases on $C_R$.   We claim that the number of arcs in $C_R\cap \M(2)$ and $C_R\cap \M(3)$ is also at least $n$ in each case. In case $n=1$ there is nothing to prove, so assume $n\ge 2$. Take two connected arcs in $\partial\Omega\cap\M(1,2)$, and assume (for a contradiction) that they determine the same connected component $\Sigma$ of $\M(2)\cap \overline{B_R}\setminus\Omega$.  For any points $x,y$, one in each of the distinct connected arcs of $\partial\Omega\cup\M(1,2)$, we may then connect them with a simple path inside $\Sigma$.  Then, that path unioned with the arc along $\partial\Omega$ connecting $x,y$, defines a bounded domain which must enclose a (connected) component of $\M(3)$.  But this contradicts our Step \ref{unbounded}, and so distinct connected arcs on $\M(2)\cap\partial\Omega$ define distinct connected (and unbounded) components of $\M(2)\cap \overline{B_R}\setminus\Omega$.  Thus the number of connected component arcs in $C_R\cap \M(2)$ must be at least $n$ (and similarly for $C_R\cap \M(3)$.) 

            Finally, to arrive at a contradiction we suppose $n\ge 2$ and make a construction suggested to us by F. Morgan. Fix a circle $C_r$ which circumscribes $\partial\Omega$.   By the previous paragraph, there are at least $n$ components of each of $\M(2)$, $\M(3)$  cut by $C_R$.  The endpoints of these arcs lie on the boundaries $\partial\M(2)$, $\partial\M(3)$, and each may be connected to distinct triple junction points on $\partial\Om$ by distinct curves forming part of the respective boundaries.  The total length of these connecting curves inside $B_R\setminus\Omega$ is thus at least $2n(R-r)$.  Now, consider the following admissible modification of $\M$:  along $C_R$, one of $\M(2)\cap C_R$ and $\M(3)\cap C_R$ can occupy at most half, $L_R\le\pi R$, of the circumference.  Suppose it is $\M(3)$.  We then create a new configuration by excising the region $\M(3)\cap B_R$ and reassigning this set to $\M(2)$.  In doing so, we have introduced a new transition between $\M(3)$ and $\M(2)$ along $C_R$, which adds $L_R$ to the total perimeter, but effectively removed the transitions between $\M(3)$ and $\M(2)$ in the region $B_R\setminus \Omega$, saving at least $2n(R-r)$ in perimeter. (See Figure~\ref{figure5} for an example.) 
            The new configuration $\tilde\M$ agrees with $\M$ outside of $C_R$, and has perimeter,
            $$ P(\tilde \M; B_R)-P(\M;B_R) \le \pi R + 2\pi r - 2n(R-r) = (\pi - 2n)R + (2\pi+2n)r <0,$$
            choosing $R$ sufficiently large, when $n\ge 2$.  This again contradicts the local minimality of $\M$, and so we must have $n=1$, and each connected component $\Omega$ of $\M(1)$ has exactly one pair of triple junctions on $\partial\Omega$.

            \begin{figure}
                \centering
                \includegraphics[scale=0.47]{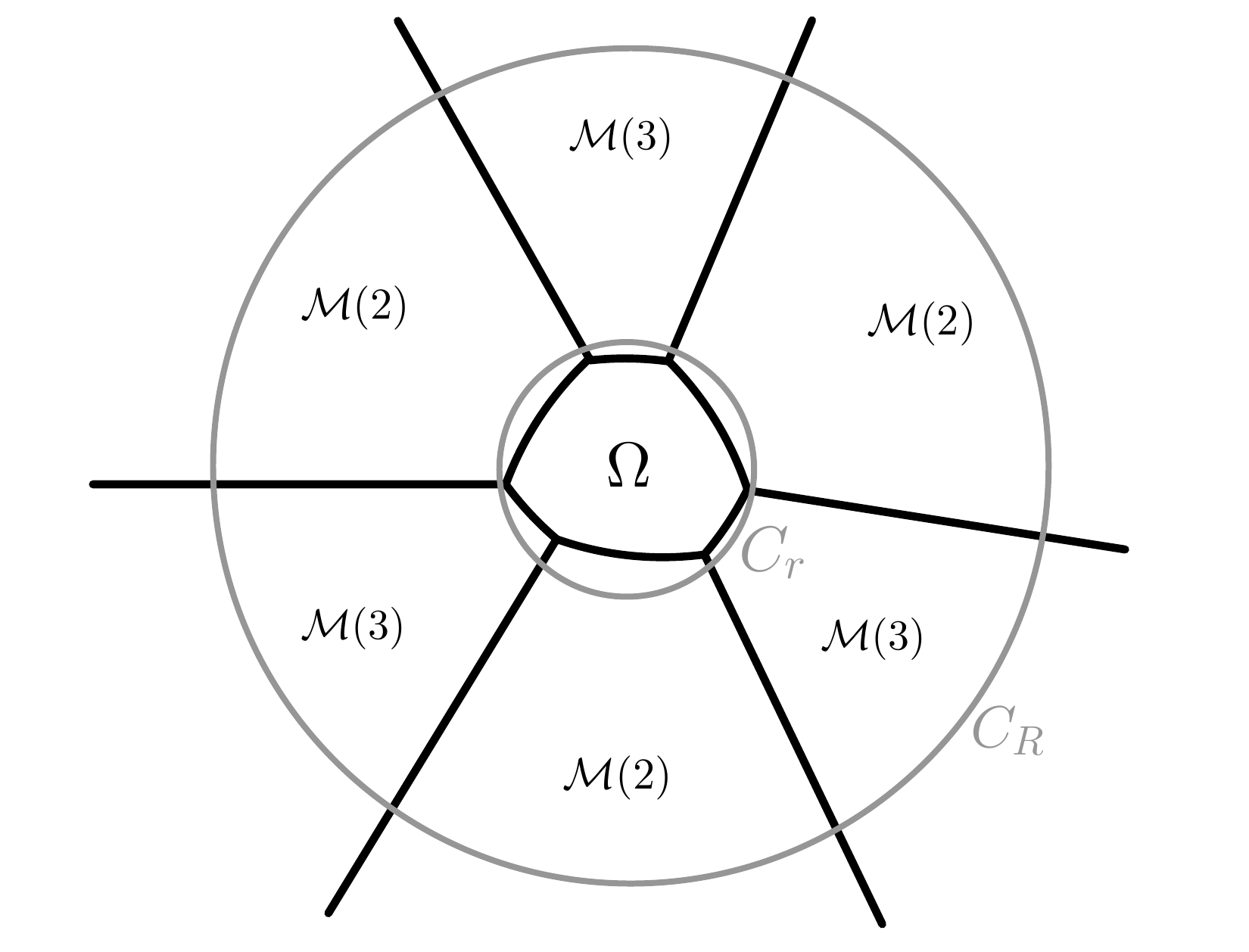}
                \caption{ $\Omega$ meets $2n$ alternating unbounded components of $\M(2)$ and $\M(3)$, and is enclosed by circles $C_R$ and $C_r$.  Reassigning the regions in $\M(3)\cap B_R$ to $\M(2)$ introduces transitions along $C_R$ but eliminates at least $2n(R-r)$ of perimeter inside $B_R$.} 
                \label{figure5}
            \end{figure}

            \item \label{isomorphic} {\it Each connected component of $\M(1)$ is isomorphic to a single lens shape (each with the same area).}  From the previous step, we know that a component $\Omega$ of $\M$ is simply connected and has exactly one pair of triple junctions, $q_1,q_2$.  Each arc in $\partial\Om$ is a circle of the same radius $s$; call $p_1,p_2$ the centers of these circles, and draw the rhombus connecting $p_1,q_1,p_2,q_2$ (each side of length $s$).  As the triple points meet with tangents of 120 degrees, the radii (which are normal to the arcs at the triple points) also meet at an angle of 120 degrees.  Hence the opposite angles to the rhombus at $p_1,p_2$ each measure 60 degrees, and the rhombus decomposes as two equilateral triangles, via the diagonal of length $s$ connecting $q_1,q_2$. The circular arcs then describe a lens, with area uniquely determined by the radius $s$, and so each connected component of $\M(1)$ is congruent.

            \item {\it There are a finite number of connected components of $\M(1)$, and it is a bounded set.} This follows directly from the finite area constraint and the previous step.  

            \item {\it Outside of a sufficiently large disk $B_R$, $\M(2)$ and $\M(3)$ coincide with complementary half-planes, and $\M(1)$ is symmetric with respect to reflection in the separating plane.}  Fix any one of the connected components $\Om$ of $\M(1)$. By Step \ref{isomorphic} we may choose axes for which $\Om$ is centered at the origin, and whose triple points lie on the $x$-axis.  By the 120 degree angle condition, a part of the transition set $\M(2,3)$ must lie along the $x$-axis outside of $\Om$.  Either these are rays extending out along the $x$-axis to infinity, or else they terminate at some other connected component of $\M(1)$.  In the latter case, those other components of $\M(1)$ are again congruent to a standard lens shape, and by the 120 degree angle condition those additional lenses must connect to $\M(2,3)$ as segments or rays along the $x$-axis.  Since there are only finitely many connected components of $\M(1)$ this process must terminate, and there is a first and last component of $\M(1)$ lying along the $x$-axis, with a component of $\E(2,3)$ lying on opposing rays along the axis.  We claim that this exhausts {\it all} connected components of $\M(1)$, and so the whole set $\M(1)$ lies along the $x$-axis, and is symmetric with respect to reflection in the axis.  Indeed, if there were another connected component $\tilde\Om$ of $\M(1)$ which did {\it not} lie on the axis, we would repeat the above procedure to reveal a distinct second line in the plane on which $\tilde\Om$ is centered.  Other components could lie on this second line, but being finite in number there will be a first and last one along the line, and another component of $\M(2,3)$ would exist as a pair of collinear rays, distinct from the $x$-axis.  However, this contradicts Step~\ref{unbounded}, which states that on any sufficiently large circle there can be only a single pair of transition points between $\M(2)$ and $\M(3)$.  In conclusion, there can be no other rays of $\M(2,3)$ other than along the $x$-axis, and thus outside of a large disk $\M(2)$ and $\M(3)$ coincide with the upper and lower half-planes, and all connected components of $\M(1)$ must be symmetrically placed along the $x$-axis.

            \item {\it $\M(1)$ consists of a single unit lens, centered on the $x$-axis.}  If there were several connected components in $\M(1)$, they would be aligned on the $x$-axis, and so each could be translated horizontally without changing the total area or perimeter, until they were in contact at their junction points.  This modified configuration coincides with $\M$ outside of a ball, and it therefore also a local minimizer.  However, the contact at the junction points would then  create singularities which contradict the 120 degree angle condition imposed by regularity theory, and so having multiple components in $\M$ is not possible.  By Step~\ref{isomorphic}, the single component is a unit lens. 

            \end{enumerate}

         This proves uniqueness. \qed

    \section{Some conjectures}

    Here we collect some known results and conjectures concerning minimizing planar $(N,M)$-clusters; see Table~\ref{table1} for a summary.

    \begin{table}\label{table1}
        \[\begin{array}{c|cccc}
            \tikz{\node[below left, inner sep=1pt] (M) {M};%
                \node[above right,inner sep=1pt] (N) {N};%
                \draw (M.north west|-N.north west) -- (M.south east-|N.south east);} & 0 & 1 & 2 & 3 \\
            \hline
            1 & \R^2 & \text{Disc} & \text{Std.\ $2$-bubble} & \text{Std.\ $3$-bubble} \\
            2 & \text{Two half-planes}^{\dag} & \text{Lens} & \text{Peanut}^* & \\
            3 & \text{Steiner partition}^{\dag} & \text{Tailor's chalk}^* & & 
        \end{array}\]
        \caption{Known and conjectured minimizing $(N,M)$-clusters in $\R^2$, for  $N, M \leq 3$. (${}^\dag$ indicates minimality among cone-like improper $M$-clusters; ${}^*$ indicates conjectured minimal clusters.)}
    \end{table}

    The case $N=0$ of {\it improper clusters} is discussed in \cite[Proposition 30.9]{maggi2012sets}:  among cone-like clusters, a pair of complementary half-planes is minimizing for $(0,2)$-clusters, and the Steiner partition (three rays meeting at 120 degree angles) is the unique  $(0,3)$-minimizer.  The case  $M = 1$ (a single  improper chamber) reduces to the  $N$-bubble problem. Results in the plane for  $N = 2$ are by Foisy et al.\ \cite{Foisy_etal}, and by Wichiramala \cite{wichiramala2004proof} for  $N = 3$.  Simpler alternative proofs using calibrations were more recently given by Lawlor  in \cite{lawlor_double} and \cite{lawlor2019perimeter}, for double and triple bubbles respectively.  Resolution of the case of  $N\le\min\{5, n+1\}$  bubbles in $\R^n$ was proven by Milman and Neeman \cite{milman2022structure}, who have additional results on quintuple-bubbles in higher dimensions \cite{milman2023plateau}.

    With the exception of the $(1,2)$ lens cluster, the cases with  $M \geq 2$ remain open.  We conjecture that the minimizer among the $(2,2)$-clusters with area vector $(1,1,\infty,\infty)$ is the peanut shape illustrated in Figure~\ref{fig:peanut-cluster}; and among $(1,3)$-clusters with area vector $(1,\infty,\infty,\infty)$, we believe the unique  minimizer should be the ``tailor's chalk'' configuration, as in Figure~\ref{fig:tailors-chalk-cluster}.

        \begin{figure}
            \centering
            \includegraphics[scale=0.45]{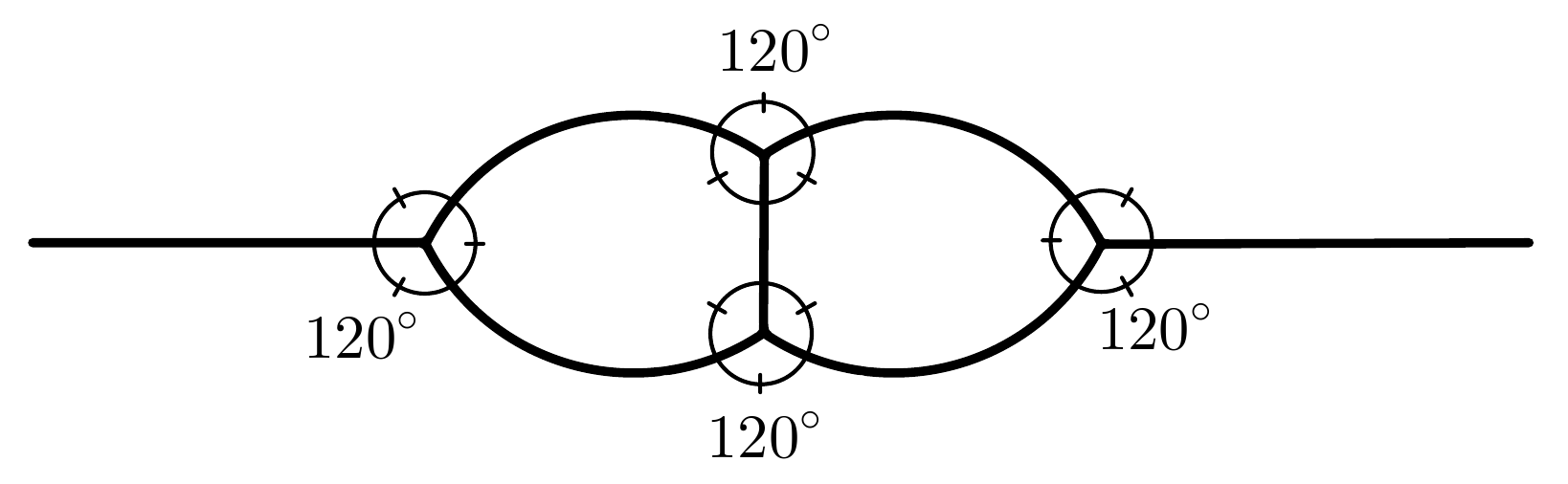}
            \caption{The peanut cluster in $\R^2$}
            \label{fig:peanut-cluster}
        \end{figure}


        \begin{figure}
            \centering
            \includegraphics[scale=0.5]{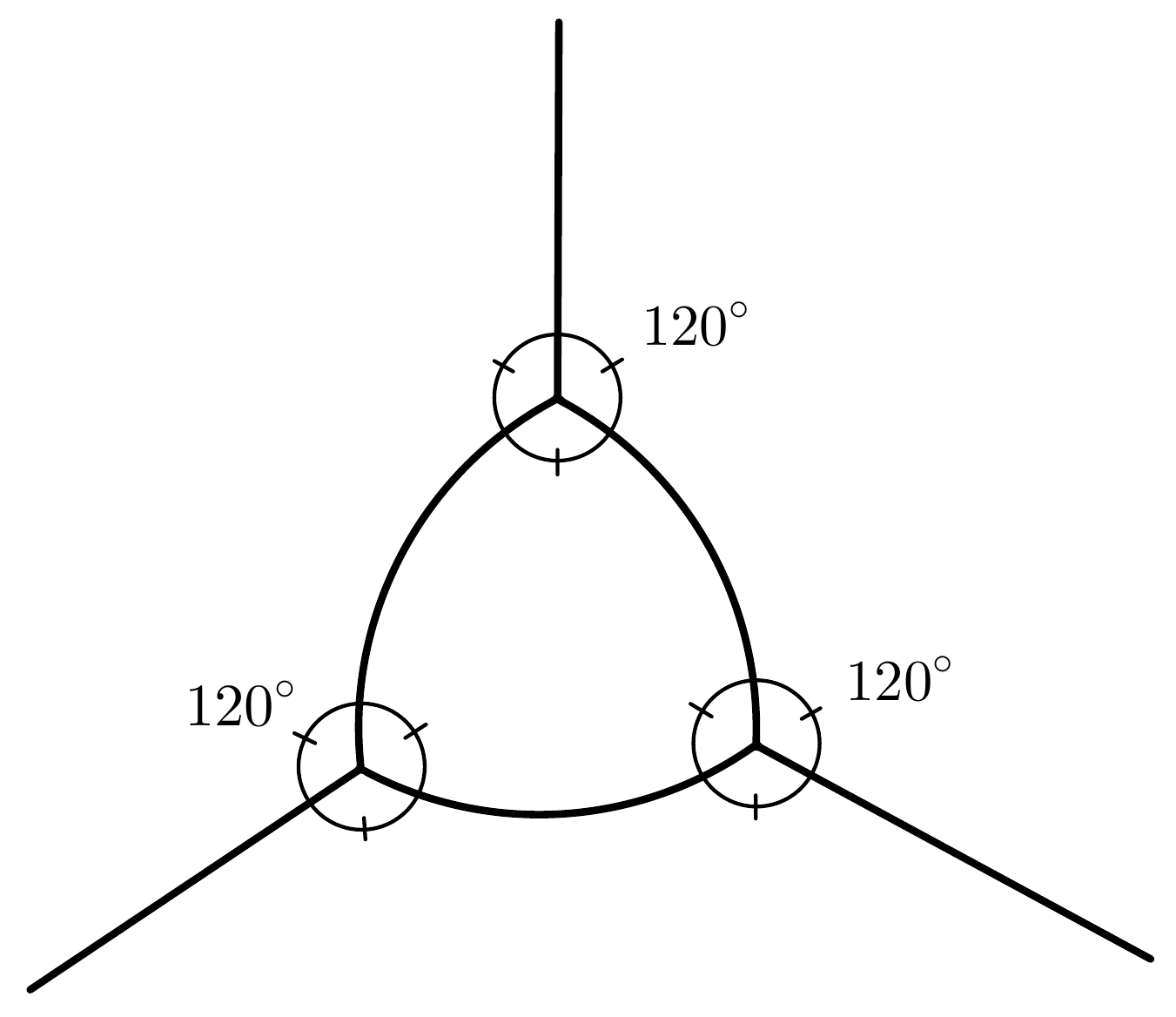}
            \caption{The tailor's chalk cluster in $\R^2$}
            \label{fig:tailors-chalk-cluster}
        \end{figure}
    

    \section{Acknowledgements}

    We would like to thank Frank Morgan, who suggested the idea for the proof of the main result, and from whom we enjoyed insightful correspondence. Thank you to Heather Lowe, whose expertise as a sewist led her to suggest the name of the tailor's chalk cluster.  SA and LB were supported by NSERC (Canada) Discovery Grants. SV was supported by an NSERC Canada Graduate Scholarship (Master's) and an Ontario Graduate Scholarship.



\bibliographystyle{plain}
   
\end{document}